\documentclass[11pt,letterpaper]{article}

\usepackage[dvips]{epsfig}
\usepackage{fancyhdr}
\usepackage{afterpage}
\usepackage{times}
\usepackage{chngcntr}
\usepackage{graphicx}
\usepackage{latexsym}
\usepackage{amsfonts}
\usepackage{amssymb}
\usepackage{amsmath}
\usepackage{amsthm}
\usepackage{bm}
\usepackage{psfrag}
\usepackage{url}
\usepackage{subfig}
\usepackage{graphicx}
\usepackage{enumitem}
\usepackage{mathdots}
\usepackage[section]{placeins}
\usepackage{graphics}
\usepackage{enumerate, color}

\newtheorem{theorem}{Theorem}[section]

\newtheorem{lemma}[theorem]{Lemma}

\newtheorem{prop}[theorem]{Proposition}

\newtheorem{property}[theorem]{Property}
\newtheorem{axiom}[theorem]{Axiom}
\newtheorem{condition}[theorem]{Condition}
\theoremstyle{definition}

\theoremstyle{remark}

\numberwithin{equation}{section}
\newtheoremstyle{dotless}{}{}{}{}{\bfseries}{}{ }{}
\theoremstyle{dotless}

\newcommand{\C}{\mathbb{C}}

\newcommand{\Q}{\mathbb{Q}}
\newcommand{\Z}{\mathbb{Z}}

\newcommand{\F}{\mathbb{F}}

\newcommand{\res}[2]{\left(\frac{#1}{#2}\right)}

\renewcommand{\tilde}{\widetilde}

\begin{document}

\author{Ian Whitehead}
\title{Affine Weyl Group Multiple Dirichlet Series: Type $\tilde{A}$}
\maketitle

\begin{abstract} 
We define a multiple Dirichlet series whose group of functional equations is the Weyl group of the affine Kac-Moody root system $\tilde{A}_n$, generalizing the theory of multiple Dirichlet series for finite Weyl groups. The construction is over the rational function field $\F_q(t)$, and is based upon four natural axioms from algebraic geometry. We prove that the four axioms yield a unique series with meromorphic continuation to the largest possible domain and the desired infinite group of symmetries.
\end{abstract}

\section{Introduction}

We will construct a multiple Dirichlet series of the form 
\begin{align} \label{series}
&Z(x_0, x_1, \ldots x_n) \nonumber \\
&=\sum_{f_0, f_1, \ldots f_n \in \F_q[t] \text{ monic}} \res{f_0}{f_1}\res{f_1}{f_2}\cdots \res{f_{n-1}}{f_n}\res{f_n}{f_0} x_0^{\deg f_0}x_1^{\deg f_1}\cdots x_n^{\deg f_n}
\end{align}
where $n\geq 2$, $q$ is a prime power, and $\res{\,}{\,}$ denotes the quadratic residue character in $\F_q[t]$. This is a new generalization of the Weyl group multiple Dirichlet series developed in papers of Brubaker-Bump-Friedberg, Chinta-Gunnells, and others. The difference is that the product of characters here is based on the following Dynkin diagram:

\includegraphics[scale=.5]{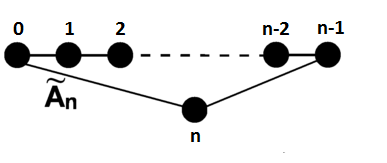}

\noindent which corresponds to the affine Kac-Moody root system $\tilde{A}_n$ rather than a finite root system. This gives the series a higher level of complexity: it will extend to a meromorphic function of $n+1$ variables with an infinite group of symmetries, the Weyl group of $\tilde{A}_n$. And it will shed light on the behavior of still-conjectural automorphic forms on the Kac-Moody Lie group. 

Multiple Dirichlet series originated as a tool to compute moments in families of L-functions. Goldfeld and Hoffstein computed the first moment of quadratic Dirichlet L-functions at the central point using a double Dirichlet series \cite{GH}; the second \cite{BFH3} and third \cite{DGH} moments have been computed using similar methods. An essential step is to replace the sums of characters like those appearing in \ref{series} with weighted sums of characters. This guarantees that the series will have a group of functional equations. The choice of weights, or correction polynomials, leads to difficult combinatorial questions, which inspired the development of Weyl group multiple Dirichlet series \cite{BBF1, BBF2, CG1, CG2}. This beautiful theory constructs, for any integer $N$ and (finite) root system $\Phi$, a multiple Dirichlet series built from $N$th power Gauss sums, whose group of functional equations is the Weyl group of $\Phi$. The series can be understood in terms of a metaplectic Casselman-Shalika formula, or as a generating function of combinatorial data on crystals associated with $\Phi$. One of the crowning achievements of this field is the Eisenstein conjecture, which connects multiple Dirichlet series back to automorphic forms; it states that each Weyl group multiple Dirichlet series appears as the constant coefficient in the Whittaker expansion of an Eisenstein series on the $N$-fold metaplectic cover of the algebraic group associated to $\Phi$. It is proven in types $A$ \cite{BBF1} and $B$ \cite{FZ}.

To go beyond the first three moments of quadratic L-functions requires multiple Dirichlet series with infinite Kac-Moody Weyl groups of functional equations. This immediately makes the series much more intricate--infinitely many symmetries imply infinitely many poles, which will accumulate at a natural boundary of meromorphic continuation. Furthermore, unlike in the finite case, the symmetries do not give sufficient information to completely pin down the series; we will show that there are infinitely many series satisfying the desired functional equations. Lee and Zhang \cite{LZ} generalize the averaging construction of Chinta and Gunnells to construct power series with meromorphic continuation to the boundary, satisfying the functional equations, for all symmetrizable Kac-Moody groups. However, their construction does not naturally contain the character sums and L-functions we would like our multiple Dirichlet series to count. Bucur and Diaconu \cite{BD}, in the special case of $\tilde{D}_4$, define a multiple Dirichlet series satisfying the functional equations by making an assumption about one of its residues. They use this series to compute the fourth moment of quadratic L-functions over rational function fields. Their construction is close in spirit to ours, but is slightly different and likely will not satisfy the Eisenstein conjecture. 

The new approach developed in forthcoming work of Diaconu and Pasol \cite{DP}, and explored in the author's thesis \cite{W}, is to construct multiple Dirichlet series axiomatically. There are four axioms \ref{twistedmult}-\ref{initialconditions}, which arise from the geometry of parameter spaces of hyperelliptic curves. Sums of characters, in particular of $\res{f_0}{f_1}\res{f_1}{f_2}\cdots \res{f_{n-1}}{f_n}\res{f_n}{f_0}$ as some of the $f_i$ range over fixed degree $a_i$ and others are held constant, can be interpreted as point counts on these parameter spaces. The weighted sums of characters introduced below are point counts after the spaces are desingularized and compactified. Axiom \ref{localglobal} is a duality statement, and Axiom \ref{dominance} is a cohomological purity statement. The axiomatic construction matches previous constructions in the case of finite root systems \cite{W}. The main theorem of this paper is that, in the case of the affine root system $\tilde{A}_n$, the four axioms produce a unique series, with meromorphic continuation to the boundary and functional equations. We expect the same theorem to hold for all affine root systems, and, if the axioms are modified as in \cite{DP}, for all Kac-Moody root systems.

Because this is a first foray into Kac-Moody multiple Dirichlet series, we have restricted our attention to type $\tilde{A}$ and to quadratic characters. It would certainly be feasible to replace these characters with $N$th power residue symbols or Gauss sums. Another restriction: our construction is over the rational function field $\F_q(t)$ only. Over number fields, proving meromorphic continuation of the analogous series is extremely difficult--in the case of $\tilde{D}_4$ over $\Q$, it is equivalent to computing the fourth moment of quadratic $L$-functions over $\Q$. However, the axioms do resolve all combinatorial problems in the construction over arbitrary global fields. The $p$-parts, or local weights, constructed in this paper will still be correct. They only need to be glued together using a different twisted multiplicativity relation for primes in the field. This means that $\F_q(t)$ plays a privileged role in the theory: it is the only field where Axioms \ref{localglobal} and \ref{dominance} hold, where the $p$-part of the series is reflected in the full series $Z$. 

The proof has three sections. First we show directly that the four axioms imply the desired functional equations of the series, and that any multivariable power series satisfying these functional equations is completely determined by its diagonal coefficients. These diagonal coefficients relate to imaginary roots in the $\tilde{A}_n$ root system, which play a subtle but critical role in the combinatorics. Next, we take a residue of the series, setting the odd-numbered $x_i$ to $q^{-1}$. We give a formula relating the diagonal residue coefficents to the original diagonal series coefficients; hence, the residue determines the full series, but unlike the series, the residue admits an Euler product formula. Balancing the effect of Axiom \ref{localglobal} on the residue coefficients against the effect of Axiom \ref{dominance} on the series coefficients, we prove the existence and uniqueness of the series. Finally, we combine this result with a close examination of the functional equations to prove an explicit formula for the residue, as a product of function field zeta functions. The meromorphic continuation of this product implies the meromorphic continuation of the full series $Z$. Only this last section does not generalize easily to other affine types. 

We end this introduction with the example of the $\tilde{A}_3$ residue formula, and its relevance to both analytic moment conjectures and the Eisenstein conjecture. We prove
\begin{align} \label{A3residue}
&q^2 \text{Res}_{x_1=x_3=q^{-1}} Z(x_0, x_1, x_2, x_3) \nonumber \\
&=\prod \limits_{m=0}^{\infty} (1-x_0^{2m+2}x_2^{2m})^{-1}(1-qx_0^{2m+2}x_2^{2m})^{-1} \nonumber \\
& \qquad \quad(1-x_0^{2m}x_2^{2m+2})^{-1}(1-qx_0^{2m}x_2^{2m+2})^{-1} \nonumber \\
& \qquad \quad (1-x_0^{2m+2}x_2^{2m+2})^{-2}(1-qx_0^{2m+2}x_2^{2m+2})^{-2} \nonumber \\ 
& \qquad \quad (1-x_0^{2m+1}x_2^{2m+1})^{-1}(1-qx_0^{2m+1}x_2^{2m+1})^{-1}.
\end{align}
We may evaluate the $\tilde{A}_3$ multiple Dirichlet series as 
\begin{align}
Z(q^{-s_0}, x, q^{-s_2}, x)&=\sum_{f_0, f_1, f_2, f_3} \res{f_1 f_3}{f_0 f_2} q^{-s_0 \deg f_0} q^{-s_2 \deg f_2} x^{\deg f_1 f_3} \nonumber \\
&=\sum_f L(s_0, \chi_f)L(s_2, \chi_f) \sigma_0(f) x^{\deg f}
\end{align}
where $f=f_1f_3$ and $\sigma_0(f)$ denotes the number of divisors of $f$. Then this residue, evaluated at $x_0=x_2=q^{-1/2}$, gives an asymptotic for the second moment of $L(1/2, \chi_f)$, weighted by $\sigma_0(f)$, over the function field $\F_q(t)$. It is possible to sieve for squarefree $f$ as well. Along with the fourth moment of $L(1/2, \chi_f)$, this is a first application of Kac-Moody root systems to number theory.

It is perhaps too early to precisely formulate an Eisenstein conjecture in the Kac-Moody setting; metaplectic Kac-Moody Eisenstein series are still conjectural, although recent work of Braverman, Garland, Kazhdan, Miller, and Patnaik makes progress constructing non-metaplectic Eisenstein series on affine Kac-Moody groups over function fields \cite{BK, BGKP, G, GMP}. In particular, a forthcoming paper of Patnaik \cite{P} gives the affine analogue of the Casselman-Shalika formula for Whittaker coefficients of these series. Crucially, Patnaik's formula contains poles corresponding to imaginary roots in the affine root system. We expect such poles to play an important role in Eisenstein series and their Whittaker coefficients for all Kac-Moody groups, including metaplectic. However, their presence cannot be detected from functional equations; different techniques are required. The series in this paper does have poles corresponding to imaginary roots. Of the factors in formula \ref{A3residue}, the first five correspond to real roots in the $\tilde{A}_3$ root system; the remaining three, however, correspond to imaginary roots. They can be compared to the factor $m$ and the factors indexed by imaginary roots $\alpha^{\vee}$ in Patnaik's formula $1.1$. Such factors do not appear in residues of the multiple Dirichlet series constructed by Lee and Zhang or Bucur and Diaconu. They provide some initial evidence that our construction is the correct one for the Eisenstein conjecture. 

\textbf{Acknowledgements}:
Great thanks to Adrian Diaconu and Dorian Goldfeld, who advised the dissertation project that led to this paper. Their wise advice and consistent support was indispensable at every stage of the process. Anna Pusk\'{a}s contributed a key idea that appears in section 5: the translation from $a$'s to $\delta$'s. I have had invaluable conversations about this project with Jeff Hoffstein, Ben Brubaker, Gautam Chinta, Paul Gunnells, Nikos Diamantis, Sol Friedberg, Kyu-Hwan Lee, Holley Friedlander, Manish Patnaik, and Jordan Ellenberg.

\section{Notation and Preliminaries}

Let $q$ be a prime power, with $q\equiv 1 \mod 4$, and let $\F_q$ be the field with $q$ elements. A polynomial in $\F_q[t]$ is called prime if it is monic and irreducible. For $f, g \in \F_q[t]$, let $\res{f}{g}$ denote the quadratic residue symbol, which is multiplicative in both $f$ and $g$. In this context, we have an extremely simple quadratic reciprocity law: for $f, g$ monic, $\res{f}{g}=\res{g}{f}$. 

We define the function field zeta function,
\begin{equation}
\zeta(x)=\sum_{f \in \F_q[t] \text{ monic}} x^{\deg f}=\prod_{p \in \F_q[t] \text{ prime}}(1-x^{\deg p})^{-1}=(1-qx)^{-1}
\end{equation}
and the quadratic Dirichlet L-function: for $g\in\F_q[t]$ monic and squarefree,
\begin{equation} 
L(x, \chi_g)=\sum_{f \in \F_q[t] \text{ monic}} \res{f}{g} x^{\deg f}=\prod_{p \in \F_q[t] \text{ prime}}(1-\res{p}{g}x^{\deg p})^{-1}.
\end{equation}
Usually, these are written as series in the variable $q^{-s}$; we have substituted the variable $x$ to emphasize the fact that these series are polynomials (or, in the case of $\zeta$, a rational function) in $x$. 

The quadratic L-functions satisfy the following functional equations: if $\deg g$ is odd, then
\begin{equation}
L(x, \chi_g) = (q^{1/2}x)^{\deg g -1} L(q^{-1}x^{-1}, \chi_g),
\end{equation}
and if $\deg g$ is even, then
\begin{equation}
(1-x)^{-1}L(x, \chi_g)= (q^{1/2}x)^{\deg g -2} (1-q^{-1}x^{-1})^{-1}L(q^{-1}x^{-1}, \chi_g).
\end{equation}
They also satisfy the Riemann hypothesis: all roots have $|x|=q^{-1/2}$.

The combinatorics of affine root systems play a critical behind-the-scenes role in the proofs below. However, for the sake of readability, we have suppressed this theory almost entirely in the exposition. Many of the type $\tilde{A}$ results here generalize to all affine types; see \cite{W}.

\section{Axioms and Functional Equations}

Let $q$ be a prime power, congruent to $1$ modulo $4$, and let $n\geq 2$ be an integer. For $f_0, f_1, \ldots f_n\in \F_q[t]$ monic, we wish to define local weights $H(f_0, f_1, \ldots f_n)\in \C$. Informally, $H(f_0, f_1, \ldots f_n)$ is a weighted version of 
\begin{equation}
\res{f_0}{f_1}\res{f_1}{f_2}\cdots\res{f_{n-1}}{f_n}\res{f_n}{f_0}
\end{equation}
where $\res{\,}{\,}$ denotes the quadratic residue symbol. The local weights determine global coefficients: for nonnegative integers $a_0, a_1, \ldots a_n$, 
\begin{equation}
c_{a_0, a_1, \ldots a_n}(q)=\sum_{\substack{f_0, f_1, \ldots f_n\in \F_q[t] \text{ monic} \\ \deg f_i=a_i}} H(f_0, f_1, \ldots f_n).
\end{equation}
We will give four axioms, originally due to Diaconu and Pasol \cite{DP}, which uniquely determine the local weights and global coefficients. These axioms describe the behavior of $c$ and $H$ as $q$ varies.

\begin{axiom}[Twisted Multiplicativity]\label{twistedmult}
Suppose that $\gcd(f_0 f_1 \cdots f_n, g_0 g_1 \cdots g_n)=1$. Then
\begin{equation}
H(f_0g_0, f_1g_1, \ldots f_ng_n)=H(f_0, f_1, \ldots f_n)H(g_0, g_1, \ldots g_n)\prod_{i \mod n+1} \res{f_i}{g_{i+1}}\res{g_i}{f_{i+1}}
\end{equation}.
\end{axiom}
With this axiom in hand, it suffices to describe $H(p^{a_0}, p^{a_1}, \ldots p^{a_n})$ for $p\in\F_q[t]$ prime. 
\begin{axiom}[Local-to-Global Principle]\label{localglobal}
The global coefficients $c_{a_0, a_1, \ldots a_n}(q)$ and local weights $H(p^{a_0}, p^{a_1}, \ldots p^{a_n})$ are polynomials in $q$ and $|p|:=q^{\deg p}$ respectively, of degrees at most $a_0+a_1+\cdots +a_n$, and 
\begin{equation}
H(p^{a_0}, p^{a_1}, \ldots p^{a_n})=|p|^{a_0+a_1+\cdots +a_n}c_{a_0, a_1, \ldots a_n}(|p|^{-1}).
\end{equation}
\end{axiom}
\begin{axiom}[Dominance Principle]\label{dominance}
The polynomial $H(p^{a_0}, p^{a_1}, \ldots p^{a_n})$ has degree less than $\frac{1}{2}(a_0+a_1+\cdots +a_n-1)$ in $|p|$; equivalently, $c_{a_0, a_1, \ldots a_n}(q)$ has terms in degrees greater than $\frac{1}{2}(a_0+a_1+\cdots +a_n+1)$. The only exceptions are for $H(1,\ldots 1)$, $H(1, \ldots 1, p, 1, \ldots 1)$, $c_{0,\ldots 0}(q)$, and $c_{0, \ldots 0, 1, 0, \ldots 0}(q)$.
\end{axiom}
The dominance principle will mean that the local weights are as small as possible under Axiom \ref{localglobal}. The final axiom is just a normalization condition:
\begin{axiom}[Initial Conditions]\label{initialconditions}
We have $H(1, \ldots 1, f_i, 1,\ldots 1)=1$ for all $f_i$, and $c_{0, \ldots 0, a_i, 0, \ldots 0}(q)=q^{a_i}$ for all $a_i$.
\end{axiom}

We define the quadratic $\tilde{A}_n$ multiple Dirichlet series over the rational function field $\F_q(t)$ as 
\begin{align}
Z(x_0,x_1,\ldots x_n):=&\sum_{f_0, f_1, \ldots f_n \in \F_q[t] \text{ monic}} H(f_0, f_1, \ldots f_n)x_0^{\deg f_0}x_1^{\deg f_1}\cdots x_n^{\deg f_n} \nonumber \\
&=\sum_{a_0, a_1, \ldots a_n \geq 0} c_{a_0, a_1, \ldots a_n}(q)x_0^{a_0}x_1^{a_1}\cdots x_n^{a_n}
\end{align}
The main theorem of this paper is as follows:
\begin{theorem}\label{main}
There exists a unique choice of local weights $H(f_0, f_1, \ldots f_n)$ and global coefficients $c_{a_0, a_1, \ldots a_n}(q)$ satisfying the four axioms. Moreover, these give rise to a multiple Dirichlet series $Z(x_0,x_1,\ldots x_n)$ with meromorphic continuation to $|x_0x_1\cdots x_n|<q^{-(n+1)/2}$ and group of functional equations isomorphic to the affine Weyl group of $\tilde{A}_n$. 
\end{theorem}
We will prove the last statement first; that is, assuming we have chosen weights and coefficients satisfying the axioms, the resulting series satisfies the functional equations. For now, these functional equations are identities of formal power series only, but they hold as identities of functions once $Z$ is meromorphically continued. We will require some additional definitions of single-variable series contained within $Z$: let
\begin{align}
&\Lambda_{a_0, \ldots \hat{a_i}, \ldots a_n}(x_i)=\sum_{a_i \geq 0} c_{a_0, \ldots a_i, \ldots a_n}(q)x_i^{a_i} \\
&\lambda_{p^{a_0}, \ldots \hat{p^{a_i}}, \ldots p^{a_n}}(x_i)=\sum_{a_i \geq 0} H(p^{a_0}, \ldots p^{a_i}, \ldots p^{a_n})x_i^{a_i \deg p} \\
&L_{f_0, \ldots \hat{f_i}, \ldots f_n}(x_i)=\sum_{f_i \in \F_q[t] \text{ monic}} H(f_1, \ldots f_i, \ldots f_n) x_i^{\deg f_i} 
\end{align}
where we use the notation $\hat{a}$ or $\hat{f}$ for an omitted index. The local series $\lambda$ can be obtained from the global series $\Lambda$ by substituting $q\mapsto |p|^{-1}$, $x_i \mapsto |p|x_i^{\deg p}$ and multiplying by $|p|^{a_0+\cdots+\hat{a_i}+\cdots+a_n}$.
\begin{prop}\label{axiomsimplyfes}
Fix all but one $a_i$, for $i$ modulo $n+1$. If $a_{i-1}+a_{i+1}$ is odd, then $\Lambda_{a_0, \ldots \hat{a_i}, \ldots a_n}(x_i)$ and $\lambda_{p^{a_0}, \ldots \hat{p^{a_i}}, \ldots p^{a_n}}(x_i)$ are polynomials of degrees $a_{i-1}+a_{i+1}-1$, $(a_{i-1}+a_{i+1}-1)\deg p$ respectively, satisfying:
\begin{align}
&(q^{1/2} x_i)^{a_{i-1}+a_{i+1}-1} \Lambda_{a_0, \ldots \hat{a_i}, \ldots a_n}(q^{-1}x_i^{-1}) = \Lambda_{a_0, \ldots \hat{a_i}, \ldots a_n}(x_i) \\ 
&(q^{1/2} x_i)^{(a_{i-1}+a_{i+1}-1)\deg p} \lambda_{p^{a_0}, \ldots \hat{p^{a_i}}, \ldots p^{a_n}}(q^{-1}x_i^{-1}) = \lambda_{p^{a_0}, \ldots \hat{p^{a_i}}, \ldots p^{a_n}}(x_i). 
\end{align}
If $a_{i-1}+a_{i+1}$ is even, then these series are rational functions with denominators $1-qx_i$, $1-x_i^{\deg p}$, and numerators of degrees $a_{i-1}+a_{i+1}$, $(a_{i-1}+a_{i+1})\deg p$ respectively, satisfying:
\begin{align}
&(q^{1/2} x_i)^{a_{i-1}+a_{i+1}}(1-x_i^{-1})\Lambda_{a_0, \ldots \hat{a_i}, \ldots a_n}(q^{-1}x_i^{-1}) = (1-qx_i)\Lambda_{a_0, \ldots \hat{a_i}, \ldots a_n}(x_i) \\ 
&(q^{1/2} x_i)^{(a_{i-1}+a_{i+1})\deg p}(1-q^{-\deg p}x_i^{-\deg p})\lambda_{p^{a_0}, \ldots \hat{p^{a_i}}, \ldots p^{a_n}}(q^{-1}x_i^{-1}) \nonumber \\
& \quad = (1-x_i^{\deg p})\lambda_{p^{a_0}, \ldots \hat{p^{a_i}}, \ldots p^{a_n}}(x_i).
\end{align}
\end{prop}
\begin{proof}
Notice that in each case, the functional equations of $\Lambda$ and $\lambda$ are equivalent. The proof is by induction on $\sum_{j\neq i}a_j$. If $\sum_{j\neq i}a_j=0$, the proposition follows from Axiom \ref{initialconditions}. For the inductive step, fix $f_0, \ldots \hat{f_i}, \ldots f_n \in \F_q[t]$ monic of degrees $a_0, \ldots \hat{a_i}, \ldots a_n$. Then, by Axiom \ref{twistedmult}, we have the following Euler product formula:
\begin{align}
&L_{f_1, \ldots \hat{f_i}, \ldots f_{n+1}}(x_i) \nonumber \\
&=\left(\prod_{j\neq i, i-1} \prod_{p|f_j} \res{p^{v_p(f_j)}}{p^{-v_p(f_{j+1})}f_{j+1}}\right) \nonumber \\ & \prod_p \left(\sum_{a_i=0}^{\infty} H(p^{v_p(f_0)},\ldots p^{a_i}, \ldots p^{v_p(f_n)}) \res{p^{a_i}}{p^{-v_p(f_{i-1}f_{i+1})}f_{j-1}f_{j+1}} x_i^{a_i \deg p}\right)
\end{align}
where $v_p(f)$ denotes the number of times $p$ divides $f$. Moreover, if we set $g$ to be the squarefree part of $f_{i-1}f_{i+1}$, then all but finitely many of the Euler factors match those of 
\begin{equation}
L(x_i, \chi_g)=\prod_p (1-\res{p}{g}x_i^{\deg p})^{-1}.
\end{equation}
The only primes whose Euler factors do not match are those where $p|f_0\cdots\hat{f_i}\cdots f_n$. At such a prime, the ratio of the Euler factors is
\begin{align}
&\mu_p(x_i)=\frac{\lambda_{p^{v_p(f_0)}, \ldots \hat{p^{v_p(f_i)}}, \ldots p^{v_p(f_n)}}(\res{p}{p^{-v_p(g)}g} x_i)}{(1-\res{p}{g}x_i^{\deg p})^{-1}} \nonumber \\
&=\left\lbrace \begin{array}{cc} \lambda_{p^{v_p(f_0)}, \ldots \hat{p^{v_p(f_i)}}, \ldots p^{v_p(f_n)}}(\pm x_i) & \text{if }v_p(f_{i-1}f_{i+1}) \text{ odd} \\ (1 \mp x_i^{\deg p}) \lambda_{p^{v_p(f_0)}, \ldots \hat{p^{v_p(f_i)}}, \ldots p^{v_p(f_n)}}(\pm x_i) & \text{if } v_p(f_{i-1}f_{i+1}) \text{ even} \end{array} \right. 
\end{align}
Assuming that the proposition holds for $\lambda_{p^{v_p(f_0)}, \ldots \hat{p^{v_p(f_i)}}, \ldots p^{v_p(f_n)}}$, then this ratio
is a polynomial of degree $d_p=(v_p(f_{i-1}f_{i+1})-1)\deg p$ if $v_p(f_{i-1}f_{i+1})$ is odd, or $v_p(f_{i-1}f_{i+1})\deg p$ if $v_p(f_{i-1}f_{i+1})$ is even, with functional equation
\begin{equation}
(q^{1/2}x_i)^{d_p} \mu_p(q^{-1}x_i^{-1})=\mu_p(x_i).
\end{equation}
Combining these local functional equations with the functional equation of $L(x_i, \chi_g)$, we obtain for $a_{i-1}+a_{i+1}$ odd:
\begin{equation}
(q^{1/2}x_i)^{a_{i-1}+a_{i+1}-1}L_{f_1, \ldots \hat{f_i}, \ldots f_{n+1}}(q^{-1}x_i^{-1})=L_{f_1, \ldots \hat{f_i}, \ldots f_{n+1}}(x_i)
\end{equation}
and for $a_{i-1}+a_{i+1}$ even:
\begin{equation}
(q^{1/2}x_i)^{a_{i-1}+a_{i+1}}(1-x_i^{-1})L_{f_1, \ldots \hat{f_i}, \ldots f_{n+1}}(q^{-1}x_i^{-1})=(1-qx_i)L_{f_1, \ldots \hat{f_i}, \ldots f_{n+1}}(x_i).
\end{equation}

By definition, we have
\begin{equation}
\Lambda_{a_0, \ldots \hat{a_i}, \ldots a_n}(x_i)=\sum_{\substack{f_0, \ldots \hat{f_i}, \ldots f_n \\ \deg f_j=a_j}} L_{f_0, \ldots \hat{f_i}, \ldots f_{n+1}}(x_i)
\end{equation}
and we may assume inductively that each $L_{f_0, \ldots \hat{f_i}, \ldots f_{n+1}}(x_i)$ satisfies the desired functional equations, except for 
\begin{align}
&L_{p^{a_0}, \ldots \hat{p^{a_i}}, \ldots p^{a_n}}(x_i) \nonumber \\
&=\left\lbrace \begin{array}{cc} \lambda_{p^{a_0}, \ldots \hat{p^{a_i}}, \ldots p^{a_n}}(x_i) & \text{if } a_{i-1}+a_{i+1} \text{ odd} \\ \frac{1-x_i}{1-qx_i} \lambda_{p^{a_0}, \ldots \hat{p^{a_i}}, \ldots p^{a_n}}(x_i) & \text{if } a_{i-1}+a_{i+1} \text{ even} \end{array} \right.
\end{align}
when $p$ is a prime of degree $1$. We conclude, in the case of $a_{i-1}+a_{i+1}$ odd:
\begin{align}
&(q^{1/2}x_i)^{a_{i-1}+a_{i+1}-1}(\Lambda_{a_0, \ldots \hat{a_i}, \ldots a_n}(q^{-1}x_i^{-1})-q\lambda_{p^{a_0}, \ldots \hat{p^{a_i}}, \ldots p^{a_n}}(q^{-1}x_i^{-1})) \nonumber \\ &=\Lambda_{a_0, \ldots \hat{a_i}, \ldots a_n}(x_i)-q\lambda_{p^{a_0}, \ldots \hat{p^{a_i}}, \ldots p^{a_n}}(x_i)
\end{align}
and in the case of $a_{i-1}+a_{i+1}$ even:
\begin{align}
&(q^{1/2}x_i)^{a_{i-1}+a_{i+1}}((1-x_i^{-1})\Lambda_{a_0, \ldots \hat{a_i}, \ldots a_n}(q^{-1}x_i^{-1})\nonumber \\
& \quad -q(1-q^{-1}x_i^{-1})\lambda_{p^{a_0}, \ldots \hat{p^{a_i}}, \ldots p^{a_n}}(q^{-1}x_i^{-1})) \nonumber \\ &=(1-qx_i)\Lambda_{a_0, \ldots \hat{a_i}, \ldots a_n}(x_i)-q(1-x_i)\lambda_{p^{a_0}, \ldots \hat{p^{a_i}}, \ldots p^{a_n}}(x_i).
\end{align}
Finally, we apply Axiom \ref{dominance}. Consider the above two functional equations as identities of power series in $x_i$ and $q$. Comparing the terms whose degree in $q$ exceeds $\frac{1}{2}(a_0+\cdots+\hat{a_i}+\cdots+a_n)$ plus half their degree in $x_i$ yields the desired functional equation for $\Lambda$. Comparing the remaining terms yields the desired functional equation for $\lambda$. 
\end{proof}

For $i$ modulo $n+1$ let $\sigma_i:\C^{n+1}\to \C^{n+1}$ be given by
\begin{equation}
(\sigma_i(x_0, x_1, \ldots x_n))_j=\left\lbrace \begin{array}{cc} q^{-1}x_i^{-1} & \text{if } j\equiv i \\ q^{1/2}x_ix_j & \text{if } j\equiv i\pm 1 \\ x_j & \text{otherwise}. \end{array}\right.
\end{equation}
These transformations generate the group 
\begin{equation}
<\sigma_i|\sigma_i^2=1, \sigma_i\sigma_{i+1}\sigma_i=\sigma_{i+1}\sigma_i\sigma_{i+1}, \sigma_i\sigma_j=\sigma_j\sigma_i \text{ for } j\neq i\pm 1>
\end{equation}
which is the $\tilde{A}_n$ affine Weyl group. Then $Z$ has functional equations
\begin{align}
& Z_{a_{i+1}+a_{i-1} \text{ odd}}(\sigma_i(x_0,\ldots x_n))=q^{1/2} x_i Z_{a_{i+1}+a_{i-1} \text{ odd}}(x_0, \ldots x_n), \label{oddfe}\\
& (1-x_i^{-1}) Z_{a_{i+1}+a_{i-1} \text{ even}}(\sigma_i(x_0, \ldots x_n))= (1-qx_i) Z_{a_{i+1}+a_{i-1} \text{ even}}(x_0, \ldots x_n) \label{evenfe}
\end{align}
where $Z_{a_{i+1}+a_{i-1} \text{ odd/even}}$ denotes sum of terms $c_{a_0,\ldots a_n}(q)x_0^{a_0}\cdots x_n^{a_n}$ with $a_{i-1}+a_{i+1}$ odd or even, respectively.

We may also define the $p$-part of $Z$,
\begin{equation}
Z_p(x_0, \ldots x_n)=\sum_{a_0, \ldots a_n} H(p^{a_0}, \ldots p^{a_n})x_0^{a_0\deg p}\cdots x_n^{a_n\deg p}
\end{equation}
and obtain similar local functional equations.

The functional equations are identities of formal power series, which may be translated into linear recurrences on the coefficients. If $a_{i-1}+a_{i+1}$ is odd, then  we have 
\begin{equation} \label{oddrecurrence}
c_{\ldots a_i, \ldots}(q)=q^{a_i-(a_{i-1}+a_{i+1}-1)/2} c_{\ldots a_{i-1}+a_{i+1}-1-a_i, \ldots}(q) 
\end{equation}
and if $a_{i-1}+a_{i+1}$ is even, then 
\begin{align} \label{evenrecurrence}
&c_{\ldots a_i, \ldots}(q) \nonumber \\ 
&=q c_{\ldots a_i-1, \ldots}(q) + q^{a_i-(a_{i-1}+a_{i+1})/2}(c_{\ldots a_{i-1}+a_{i+1}-a_i, \ldots}(q)-q c_{a_{i-1}+a_{i+1}-a_i-1}(q)).
\end{align}
Note that any coefficient with $a_i > \frac{1}{2}(a_{i-1}+a_{i+1})$ can be rewritten in terms of lower coefficients. The only undetermined coefficients are the diagonals, $c_{a, a, \ldots a}(q)$.

On the other hand, the transformations $\sigma_i$ leave the form $x_0x_1\cdots x_n$ invariant, so we may multiply $Z$ by an arbitrary power series in this variable, thereby obtaining an arbitrary family of diagonal coefficients, without affecting the functional equations. We have proved the following:

\begin{prop}\label{diagonal}
A series $Z(x_0, \ldots x_n)=\sum\limits_{a_0, \ldots a_n} c_{a_0, \ldots a_n}(q)x_0^{a_0}\cdots x_n^{a_n}$ which satisfies the functional equations \ref{oddfe} and \ref{evenfe} is uniquely determined by its diagonal coefficients $c_{a, \ldots a}(q)$. 
\end{prop}

In fact, by this proposition, the ratio of two power series satisfying the functional equations must be a diagonal series. By inspecting the recursive formulas, one also sees that if the  $c_{a, \ldots a}(q)$ satisfy Axiom \ref{dominance}, then so do all the coefficients.

The last result of this section guarantees the existence of a meromorphic power series satisfying the functional equations. It is proven by a generalization of the averaging procedure developed by Chinta and Gunnells \cite{CG1}. A proof for the affine root system $\tilde{D}_4$ appears in Bucur-Diaconu \cite{BD}; Lee and Zhang \cite{LZ} show that the construction works in all symmetrizable Kac-Moody root systems. The notation of Lee and Zhang differs somewhat from ours; in particular, they construct the $p$-part $Z_p$ of a series--a change of variables is required to obtain the global series $Z$. For a proof in the notation of this paper, see \cite{W}.

We construct an infinite product which describes all the poles of $Z$ implied by the functional equations: let
\begin{equation}
\Delta(x_0, \ldots x_n)=\prod_{m=0}^{\infty} \prod_{\substack{i, j \mod n+1 \\ i\not\equiv j+1}} (1-q(qx_0^2\cdots qx_n^2)^m(qx_i^2\cdots qx_j^2)).
\end{equation}
$\Delta$ is best understood as a deformed Weyl denominator, a product over positive real roots in the $\tilde{A}_n$ root system. It converges for $|x_0\cdots x_n|<q^{-(n+1)/2}$. Then we have:

\begin{prop} \label{averaging}
There exists a power series $Z_{\text{avg}}(x_0, \ldots x_n)$ satisfying the functional equations \ref{oddfe} and \ref{evenfe}, such that $\Delta(x_0, \ldots x_n)Z_{\text{avg}}(x_0, \ldots x_n)$ has analytic continuation to $|x_0\cdots x_n|<q^{-(n+1)/2}$. 
\end{prop} 

The series $Z_{\text{avg}}$ does not satisfy our axioms, but it will be crucial in proving the meromorphic continuation of our series $Z$.

\section{Existence and Uniqueness}
To simplify computations leading to the proof of the main theorem, we restrict our consideration to a particular residue $R$ of the series $Z$. Such residues are essential when we apply Tauberian theorems to $Z$ to obtain analytic results. The series $Z$ can be recovered from the residue $R$, but $R$ is simpler than $Z$ because its coefficients are multiplicative, not twisted multiplicative; hence, $R$ has an Euler product. We will observe a symmetry in the Euler product formula which is equivalent to the local-to-global axiom. This, together with the dominance axiom, leads to an explicit formula for $R$. The meromorphic continuation of $R$ implies that of $Z$. 

We will have two separate cases: $n$ odd and $n$ even. The computation for $n$ odd is more straightforward; for $n$ even, the analogous results are somewhat more complicated, but not essentially different.

For $n$ odd, we define 
\begin{equation}
R(x_0, x_2, \ldots x_{n-1})=(-q)^{(n+1)/2} \text{Res}_{x_1=x_3=\cdots=x_n=q^{-1}} Z(x_0, \ldots x_n) 
\end{equation}
and for $n$ even,
\begin{equation}
R(x_0, x_2, \ldots x_n)=(-q)^{n/2} \text{Res}_{x_1=x_3=\cdots=x_{n-1}=q^{-1}} Z(q^{-1/4}x_0, x_1, \ldots x_{n-1}, q^{-1/4}x_n) 
\end{equation}
where the first and last variables are multiplied by $q^{-1/4}$ to simplify later formulas. At first, we treat this residue as a formal power series only. Taking a residue may not be a well-defined operation on arbitrary power series, but in this case it is. We may multiply $Z(x_0, \ldots x_n)$ by $1-qx_i$ and then evaluate at $x_i=q^{-1}$; by Proposition \ref{axiomsimplyfes} this involves taking only finite sums of coefficients, so it gives a well-defined series. This is the meaning of $-q\text{Res}_{x_i=q^{-1}} Z(x_0, \ldots x_n)$. Once we give an explicit formula for $R$ in the following chapter, we will see that it is a meromorphic function, and the residue of a meromorphic function $Z$, in the desired domain.

If $Z$ and $Z'$ are two series satisfying the functional equations \ref{oddfe} and \ref{evenfe}, with residues $R$ and $R'$, then by Proposition \ref{diagonal}, 
\begin{equation}
\frac{Z(x_0, \ldots x_n)}{Z'(x_0, \ldots x_n)}=Q(x_0\cdots x_n)
\end{equation}
a diagonal series. Therefore,
\begin{equation}
\frac{R(x_0, x_2,\ldots x_{2\lfloor n/2 \rfloor})}{R'(x_0, x_2,\ldots x_{2\lfloor n/2 \rfloor})}=Q(q^{-(n+1)/2}x_0x_2\cdots x_{2\lfloor n/2 \rfloor})
\end{equation}
And the same holds if we only take the diagonal parts of the power series: given $Z(x_0, \ldots x_n)=\sum_{a_0,\ldots a_n}c_{a_0, \ldots a_n}(q)x_0^{a_0}\cdots x_n^{a_n}$, we define $Z_{\text{diag}}(x)=\sum_a c_{a, \ldots a}(q) x^a$, and similarly for $Z'_{\text{diag}}(x)$, $R_{\text{diag}}(x)$, and $R'_{\text{diag}}(x)$. Then 
\begin{equation}
\frac{Z_{\text{diag}}(x)}{Z'_{\text{diag}}(x)}=\frac{R_{\text{diag}}(q^{(n+1)/2}x)}{R'_{\text{diag}}(q^{(n+1)/2}x)}=Q(x).
\end{equation}
Equivalently, the ratio 
\begin{equation}
\frac{R_{\text{diag}}(x)}{Z_{\text{diag}}(q^{-(n+1)/2}x)}=:P(x)
\end{equation}
is the same for all series satisfying the functional equations \ref{oddfe} and \ref{evenfe}. It does not depend on the choice of diagonal coefficients $c_{a, \ldots a}(q)$. $P(x)$ can be thought of as the diagonal part of the residue if we take $c_{0,\ldots 0}(q)=1$ and $c_{a,\ldots a}(q)=0$ for all $a>0$. 

Thus we can recover a full series $Z(x_0, \ldots x_n)$ satisfying the functional equations from its residue $R$, or even from $R_{\text{diag}}$. 

Next, we prove two formulas for the coefficients of $R$ in terms of the coefficients of $Z$:
\begin{prop} \label{residueformula1}
Suppose the coefficients of $Z(x_0, \ldots x_n)$ are $c_{a_0, \ldots a_n}(q)$. Then for $n$ odd, the coefficient of $R(x_0, x_2, \ldots x_{n-1})$ at $x_0^{a_0}x_2^{a_2}\cdots x_{n-1}^{a_{n-1}}$ is 
\begin{equation}
q^{-2(a_0+a_2+\cdots +a_{n-1})} c_{a_0, a_0+a_2, a_2, a_2+a_4, \ldots a_{n-1}, a_{n-1}+a_0}(q).
\end{equation}
For $n$ even, the coefficient at $x_0^{a_0}x_2^{a_2}\cdots x_{n}^{a_{n}}$ is 
\begin{equation}
q^{3(a_0+a_n)/4-2(a_0+a_2+\cdots+a_n)}c_{a_0, a_0+a_2, a_2, a_2+a_4, \ldots a_{n-2}+a_n, a_n}(q).
\end{equation}
In particular, the nonzero coefficients of the residue must have all $a_i$ odd or all $a_i$ even.
\end{prop}
\begin{proof}
Fix all indices except for one $a_i$ with $i$ odd. Then by Proposition \ref{axiomsimplyfes}, we have 
\begin{align}
&\Lambda_{a_0, \ldots \hat{a_i}, \ldots a_n}(x_i) \nonumber \\
&= \sum_{a_i=0}^{a_{i-1}+a_{i+1}-1} c_{\ldots a_{i-1}, a_i, a_{i+1} \ldots}(q) x_i^{a_i} + \frac{c_{\ldots a_{i-1}, a_{i-1}+a_{i+1}, a_{i+1} \ldots}(q) x_i^{a_{i-1}+a_{i+1}}}{1-qx_i}
\end{align}
and $c_{\ldots a_{i-1}, a_{i-1}+a_{i+1}, a_{i+1} \ldots}(q)=0$ if $a_{i-1}+a_{i+1}$ is odd. 
Taking $-q \text{Res}_{x_i=q^{-1}}$ gives $q^{-a_{i-1}-a_{i+1}}c_{\ldots a_{i-1}, a_{i-1}+a_{i+1}, a_{i+1} \ldots}(q)$. If we repeat this process for all $i\leq n$ odd, we obtain the desired formula. The rearrangements of power series implicit in this computation are only reorderings of finite sums, by Proposition \ref{axiomsimplyfes}.
\end{proof}

We may apply the functional equations $\sigma_i$ for $i \leq n$ odd to the residue. Since only terms with $a_{i-1}+a_{i+1}$ even contribute to the residue, we only need the even part of the functional equation. Let $Z_\text{same}$ denote the part of the power series $Z$ with $a_0+a_2$, $a_2+a_4$, $a_4+a_6$, etc. all even, i.e. where $a_0$, $a_2$, $a_4$, etc. have the same parity. In the case of $n$ odd, we obtain:
\begin{align} \label{residuevalue1}
R(x_0, x_2, \ldots x_{n-1}) &=(-q)^{(n+1)/2} \text{Res}_{x_1=x_3=\cdots=x_n=q^{-1}} Z(x_0, \ldots x_n) \nonumber \\
&=(-q)^{(n+1)/2} \text{Res}_{x_1=x_3=\cdots=x_n=q^{-1}} Z_{\text{same}}(x_0, \ldots x_n) \nonumber \\
&=(1-q)^{(n+1)/2}Z_{\text{same}}(q^{-1}x_0,1, q^{-1}x_2, 1, \ldots q^{-1}x_{n-1}, 1)
\end{align}
and for $n$ even:
\begin{equation} \label{residuevalue2}
R(x_0, x_2, \ldots x_n)=(1-q)^{n/2}Z_{\text{same}}(q^{-3/4} x_0,1, q^{-1}x_2, 1, \ldots q^{-1}x_{n-2}, 1, q^{-3/4}x_n).
\end{equation}
We can now prove a second formula for the coefficients of $R$ in terms of the local coefficients $H(f_0, \ldots f_n)$ rather than the global coefficients $c_{a_0, \ldots a_n}(q)$:

\begin{prop} \label{residueformula2}
If $Z$ is a series defined by local weights $H$ as:
\begin{equation}
Z(x_0, \ldots x_n)=\sum_{f_0, \ldots f_n} H(f_0, \ldots f_n) x_0^{\deg f_0}\cdots x_n^{\deg f_n}
\end{equation}
then for $n$ odd we have:
\begin{align}
&R(x_0, x_2, \ldots x_{n-1}) \nonumber \\
&=\sum_{f_0, f_2, \ldots f_{n-1}} \frac{H(f_0, f_0f_2, f_2, f_2f_4,\ldots f_{n-1}, f_{n-1}f_0)}{q^{\deg f_0+\deg f_2 + \cdots +\deg f_{n-1}}} x_0^{\deg f_0}x_2^{\deg f_2}\cdots x_{n-1}^{\deg f_{n-1}}
\end{align}
and for $n$ even:
\begin{align}
&R(x_0, x_2, \ldots x_n) \nonumber \\
&=\sum_{f_0, f_2, \ldots f_n} \frac{H(f_0, f_0f_2, f_2, f_2f_4,\ldots f_{n-2}f_n, f_n)}{q^{\frac{3}{4}\deg f_0+\deg f_2 + \cdots +\deg f_{n-2}+\frac{3}{4}\deg f_n}} x_0^{\deg f_0}x_2^{\deg f_2}\cdots x_n^{\deg f_n}.
\end{align}
In particular, only tuples of polynomials $f_0, f_2, f_4,\ldots$ with $f_i f_{i+2}$ a perfect square for all $i$, i.e. tuples of polynomials with the same squarefree part, contribute to the residue. 
\end{prop}
\begin{proof}
First note that $H(f_0, f_0f_2, f_2, f_2f_4,\ldots)$ indeed vanishes if any $f_if_{i+2}$ is not a perfect square. If a prime $p$ divides $f_if_{i+2}$ an odd number of times, then $H(p^{v_p(f_0)}, p^{v_p(f_0f_2)}, p^{v_p(f_2)}, p^{v_p(f_2f_4)},\ldots)$ must vanish by the local functional equations. Then, by Axiom \ref{twistedmult}, $H(f_0, f_0f_2, f_2, f_2f_4,\ldots)=0$.

Now we use equations \ref{residuevalue1} and \ref{residuevalue2} as a starting point. Fix all but one $f_i$ and assume $\deg f_{i-1}f_{i+1}$ is even. Then, as in the proof of Proposition \ref{axiomsimplyfes}, the series $\sum_{f_i} H(f_0, \ldots f_n)x_i^{\deg f_i}$ matches the L-function $L(x_i, \chi_{f_{i-1}f_{i+1}})$ up to multiplication by correction polynomials. Unless $f_{i-1}f_{i+1}$ is a perfect square, this L-function has a trivial zero at $x_i=1$. If $f_{i-1}f_{i+1}$ is a perfect square, then this series matches the zeta function $(1-qx_i)^{-1}$ up to multiplication by correction polynomials of the form 
\begin{align}
(1-x_i^{\deg p})\left(\sum_{a_i=0}^{v_p(f_{i-1}f_{i+1})-1} H(\ldots p^{v_p(f_{i-1})},p^{a_i}, p^{v_p(f_{i+1})}, \ldots) x_i^{a_i\deg p}\right) \nonumber\\
+ H(\ldots p^{v_p(f_{i-1})}, p^{v_p(f_{i-1}f_{i+1})}, p^{v_p(f_{i+1})}, \ldots) x_i^{v_p(f_{i-1}f_{i+1})\deg p}
\end{align}
for each prime $p$ dividing $f_{i-1}f_{i+1}$. Thus, evaluating the series $\sum_{f_i} H(f_0, \ldots f_n)x_i^{\deg f_i}$ at $x_i=1$ and multiplying by $1-q$ gives $H(\ldots f_{i-1}, f_{i-1}f_{i+1}, f_{i+1},\ldots)$. Repeating this process for all $i\leq n$ odd and using equations \ref{residuevalue1} and \ref{residuevalue2} gives the desired result. Again by Proposition \ref{axiomsimplyfes}, the rearrangements of power series implicit in this proof are only reorderings of finite sums. 
\end{proof}

Notice that the terms $H(f_0, f_0f_2, f_2, f_2f_4,\ldots)$ appearing in the residue are multiplicative, not twisted multiplicative. That is, if $\gcd(f_0f_2\cdots f_{2\lfloor n/2\rfloor}, g_0g_2\cdots g_{2\lfloor n/2\rfloor})=1$, then 
\begin{equation}
H(f_0g_0, f_0f_2g_0g_2, f_2g_2, f_2f_4g_2g_4,\ldots)=H(f_0, f_0f_2, f_2, f_2f_4,\ldots)H(g_0, g_0g_2, g_2, g_2g_4,\ldots).
\end{equation}
Indeed, since all the $f_if_{i+2}$ and $g_ig_{i+2}$ terms are square, they do not contribute to the twists in Axiom \ref{twistedmult}. The only possible factor of $-1$ comes from $\res{f_0}{g_n} \res{f_n}{g_0}$ in the case when $n$ is even, but since $f_0$ and $f_n$ must have the same squarefree part, and so must $g_0$ and $g_n$, this product of residues is $1$.

Therefore, we may write an Euler product expression for $R$:
\begin{equation}
R(x_0, x_2, \ldots x_{2\lfloor n/2 \rfloor})=\prod_{p \text{ prime}} R_p(x_0, x_2, \ldots x_{2\lfloor n/2 \rfloor})
\end{equation}
where if $n$ is odd,
\begin{align} 
&R_p(x_0, x_2, \ldots x_{n-1}) \nonumber\\ & \sum_{a_0, a_2, \ldots a_{n-1}} \frac{H(p^{a_0}, p^{a_0+a_2}, p^{a_2}, p^{a_2+a_4}, \ldots p^{a_{n-1}}, p^{a_{n-1}+a_0})}{q^{(a_0+a_2+\cdots+a_{n-1})\deg p}} x_0^{a_0\deg p} x_2^{a_2\deg p}\cdots x_{n-1}^{a_{n-1}\deg p}
\end{align}
and if $n$ is even, 
\begin{align} 
&R_p(x_0, x_2, \ldots x_n) \nonumber\\ &= \sum_{a_0, a_2, \ldots a_n} \frac{H(p^{a_0}, p^{a_0+a_2}, p^{a_2}, p^{a_2+a_4}, \ldots p^{a_{n-2}+a_n}, p^{a_n})}{q^{(\frac{3}{4}a_0+a_2+\cdots+a_{n-2}+\frac{3}{4}a_n)\deg p}} x_0^{a_0\deg p} x_2^{a_2\deg p}\cdots x_n^{a_n\deg p}.
\end{align}
Let us compare these equations to Proposition \ref{residueformula1}, and apply Axiom \ref{localglobal}. We see that $R_p$ may be obtained from $R$ by substituting $q\mapsto q^{-\deg p}$ and $x_i\mapsto x_i^{\deg p}$.

We may write $R(x_0, x_2, \ldots)$ as a product of terms $(1-q^{\beta}x_0^{\alpha_0}x_2^{\alpha_2}\cdots )^{-\gamma}$, where $\alpha_i\in \Z$, and $\beta\in \Z$ for $n$ odd, or $\frac{1}{2}\Z$ for $n$ even (not $\frac{1}{4}\Z$ because $\alpha_0, \alpha_n$ must have the same parity). Any formal power series in can be expressed uniquely in this way. Then comparing to the Euler product formula for the function field zeta function, we conclude that such factors come in pairs.

\begin{property} \label{symmetry} If $(1-q^{\beta}x_0^{\alpha_0}x_2^{\alpha_2}\cdots )^{-\gamma}$ is a factor of $R$, then so is $(1-q^{1-\beta}x_0^{\alpha_0}x_2^{\alpha_2}\cdots )^{-\gamma}$.
\end{property}

This symmetry is equivalent to Axiom \ref{localglobal} for the full series $Z$: it implies the local-to-global property for coefficients $c_{a_0, a_0+a_2, a_2, a_2+a_4,\ldots}(q)$ and local weights $H(f_0, f_0f_2, f_2, f_2f_4, \ldots)$, and all other coefficients can be obtained from these via compatible global and local functional equations.  

We are now ready to prove the existence and uniqueness statements of the main theorem. We first introduce several new notations based on the expression of $R$ as a product of terms $(1-q^{\beta}x_0^{\alpha_0}x_2^{\alpha_2}\cdots )^{-\gamma}$. We let $R^{\flat}$ denote the product of terms with $\beta \leq 0$ and $R^{\sharp}$ denote the product of terms with $\beta \geq 1$. If $n$ is even, we also have $R^{\natural}$ with $\beta=\frac{1}{2}$. We let $R_1$ denote the product of diagonal factors (with $\alpha_0=\alpha_2=\alpha_4=\cdots$), and let $R_0$ denote the product of off-diagonal factors. In the following chapter, we will explicitly compute $R_0$, and show that it satisfies property \ref{symmetry}. For now, we assume this.  We also have $R_0^{\flat}$, $R_0^{\natural}$, $R_0^{\sharp}$, $R_1^{\flat}$, $R_1^{\natural}$, $R_1^{\sharp}$, and the diagonal parts of each of these. 

Recall that 
\begin{equation}
P(x)Z_{\text{diag}}(q^{-(n+1)/2}x)=R_{\text{diag}}(x)=R_{0, \text{ diag}}(x)R_1(x) \label{diagonaltoresidue}
\end{equation}
for any series $Z$ satisfying the functional equations with residue $R$. We must show that there exists a unique choice of $Z_{\text{diag}}$ satisfying Axiom \ref{dominance} and $R_1$ satisfying Property \ref{symmetry} which make this equation true. The resulting series $Z$ with residue $R$ will satisfy the four axioms.

For $Z_{\text{diag}}(q^{-(n+1)/2}x)$ to satisfy the dominance axiom, its coefficients (other than the constant coefficient $1$) must be polynomials divisible by $q$. If it is written as a product of terms $(1-q^{\beta}x^{\alpha})^{-\gamma}$, then all these terms will have $\beta \geq 1$. Thus, when $n$ is odd, 
\begin{align} \label{flat}
R_1^{\flat}(x)&=(P(x)Z_{\text{diag}}(q^{-(n+1)/2}x)R_{0, \text{ diag}}(x)^{-1})^{\flat} \nonumber \\
&=(P(x)R_{0, \text{ diag}}(x)^{-1})^{\flat}
\end{align}
and both $P(x)$ and $R_{0, \text{ diag}}(x)^{-1}$ are fixed. Hence, $R_1^{\flat}$ is uniquely determined. Then we must choose $R_1^{\sharp}$ to satisfy Property \ref{symmetry}, and $Z_{\text{diag}}$ to satisfy equation \ref{diagonaltoresidue}. In the case of $n$ even, $R_1^{\flat}R_1^{\natural}$ is uniquely determined, and the conclusion is the same as before. 

\section{Computing the Residue}

In this chapter, we prove explicit formulas for the residue.

\begin{prop} \label{Rformula}
If the series $Z(x_0, \ldots x_n)$ satisfies Axioms \ref{twistedmult}-\ref{initialconditions}, and $n$ is odd, then the residue $R(x_0, x_2, x_{n-1})$ is as follows:
\begin{align}\label{residueodd}
&R(x_0, x_2, \ldots x_{n-1}):=(-q)^{(n+1)/2} \text{Res}_{x_1=x_3=\cdots=x_n=q^{-1}} Z(x_0, \ldots x_n) \nonumber \\ 
&=\prod_{m=0}^{\infty} (1-(x_0 x_2\cdots x_{n-1})^{2m+1})^{-1}(1-q (x_0 x_2\cdots x_{n-1})^{2m+1})^{-1} \nonumber\\ 
& \qquad \prod_{\substack{i,j \text{ mod } n+1 \\ \text{even}}} (1-(x_0 x_2\cdots x_{n-1})^{2m}(x_i x_{i+2}\cdots x_j)^2)^{-1} \nonumber \\
&\hspace{6em} (1-q (x_0 x_2\cdots x_{n-1})^{2m}(x_i x_{i+2}\cdots x_j)^2)^{-1}
\end{align}
and if $n$ is even, then:
\begin{align} \label{residueeven}
&R(x_0, x_2, \ldots x_n):=(-q)^{(n+1)/2} \text{Res}_{x_1=x_3=\cdots=x_n=q^{-1}} Z(q^{-1/4}x_0, x_1, \ldots x_{n-1}, q^{-1/4}x_n) \nonumber \\
&=\prod_{m=0}^{\infty} (1-(x_0x_2\cdots x_n)^{2m+2})^{-n/2}(1-q(x_0x_2\cdots x_n)^{2m+2})^{-n/2} \nonumber \\
& \qquad \quad \prod_{\substack{0\leq i<n \\ \text{even}}} (1-q^{1/2}(x_0x_2\cdots x_n)^{2m}(x_0x_2\cdots x_i)^2)^{-1} \nonumber \\
&\hspace{6em} (1-q^{1/2}(x_0x_2\cdots x_n)^{2m}(x_{i+2}x_{i+4}\cdots x_n)^2)^{-1} \nonumber \\
&\qquad \quad \prod_{\substack{0< i\leq j <n \\ \text{even}}} (1-(x_0x_2\cdots x_n)^{2m}(x_ix_{i+2}\cdots x_j)^2)^{-1} \nonumber \\
&\hspace{6em} (1-q(x_0x_2\cdots x_n)^{2m}(x_ix_{i+2}\cdots x_j)^2)^{-1} \nonumber \\
&\hspace{6em} (1-(x_0x_2\cdots x_n)^{2m}(x_0x_2\cdots x_{i-2} x_{j+2} x_{j+4}\cdots x_n)^2)^{-1} \nonumber \\ 
&\hspace{6em} (1-q(x_0x_2\cdots x_n)^{2m}(x_0x_2\cdots x_{i-2} x_{j+2} x_{j+4}\cdots x_n)^2)^{-1}.
\end{align}
\end{prop}

These products define a meromorphic functions in the domain $|x_0x_2\cdots x_{2\lfloor n/2 \rfloor}|<1$. By Proposition \ref{averaging} we have a series $Z_{\text{avg}}$, satisfying the same functional equations as $Z$, with meromorphic continuation to $|x_0x_1\cdots x_n|<q^{-(n+1)/2}$. Its residue $R_{\text{avg}}$ is meromorphic in the same domain as $R$. The ratio 
\begin{equation}
\frac{R(x_0, x_2, \ldots x_{2\lfloor n/2 \rfloor})}{R_{\text{avg}}(x_0, x_2, \ldots x_{2\lfloor n/2 \rfloor})}=Q(q^{-(n+1)/2}x_0x_2\cdots x_{2\lfloor n/2 \rfloor})
\end{equation}
is a power series in one variable $x_0x_2\cdots x_{2\lfloor n/2 \rfloor}$, and $Q(x)$ is meromorphic for $|x|<q^{-(n+1)/2}$. Thus 
\begin{equation}
Z(x_0,\ldots x_n)=Z_{\text{avg}}(x_0, \ldots x_n)Q(x_0\cdots x_n)
\end{equation}
is meromorphic for $|x_0x_1\cdots x_n|<q^{-(n+1)/2}$. Hence, if we can prove Proposition \ref{Rformula} we will complete the proof of Theorem \ref{main}.

First we compute the residue up to a diagonal factor, using functional equations. Let $i$ be even, and, if $n$ is even, $0<i<n$. The functional equations $\sigma_i \sigma_{i-1} \sigma_{i+1} \sigma_i$ yield:
\begin{align}
&Z(\ldots x_{i-2}, x_{i-1}, x_i, x_{i+1}, x_{i+2}, \ldots) \nonumber \\
&=\frac{1}{16} \sum_{\epsilon_1, \epsilon_2, \epsilon_3, \epsilon_4 = \pm 1} \epsilon_2 \epsilon_3 q^{-2} x_{i-1}^{-2} x_i^{-4}x_{i+1}^{-2}(\epsilon_4 q^{-1/2}-\frac{1-\epsilon_2 \epsilon_3 q x_{i-1} x_i x_{i+1}}{1-\epsilon_2 \epsilon_3 q^2 x_{i-1} x_i x_{i+1}}) \nonumber \\ 
&(\epsilon_3 q^{-1/2}-\frac{1-\epsilon_1 q^{1/2} x_i x_{i+1}}{1-\epsilon_1 q^{3/2} x_i x_{i+1}}) (\epsilon_2 q^{-1/2}-\frac{1-\epsilon_1 q^{1/2} x_{i-1} x_i}{1-\epsilon_1 q^{3/2} x_{i-1} x_i}) (\epsilon_1 q^{-1/2}-\frac{1-x_i}{1-q x_i}) \nonumber \\ 
&Z(\ldots \epsilon_2 q x_{i-2} x_{i-1} x_i, \epsilon_1 \epsilon_2 \epsilon_3 \epsilon_4 x_{i-1}, \frac{\epsilon_2 \epsilon_3}{q^2 x_{i-1} x_i x_{i+1}}, \epsilon_1 \epsilon_2 \epsilon_3 \epsilon_4 x_{i+1}, \epsilon_3 q x_i x_{i+1} x_{i+2}, \ldots).
\end{align}
If we take the residue, only the terms with $\epsilon_1 \epsilon_2 \epsilon_3 \epsilon_4=1$ will contribute, and, by Proposition (\ref{residueformula1}), only even powers of $\epsilon_2$ and $\epsilon_3$ will appear. Thus the residue in fact has functional equations with scalar cocycle:
\begin{equation} \label{resfe}
R(\ldots x_{i-2}, x_i, x_{i+2}, \ldots) = (*) R(\ldots x_{i-2} x_i, \frac{1}{x_i}, x_i x_{i+2}, \ldots) 
\end{equation}
where 
\begin{align}
&(*)=\frac{1}{16} \sum_{\epsilon_1, \epsilon_2, \epsilon_3, = \pm 1} \epsilon_2 \epsilon_3 q^{2} x_i^{-4}(\epsilon_1 \epsilon_2 \epsilon_3 q^{-1/2}-\frac{1-\epsilon_2 \epsilon_3 q^{-1}x_i}{1-\epsilon_2 \epsilon_3 x_i}) \nonumber \\ 
&(\epsilon_3 q^{-1/2}-\frac{1-\epsilon_1 q^{-1/2} x_i}{1-\epsilon_1 q^{1/2} x_i}) (\epsilon_2 q^{-1/2}-\frac{1-\epsilon_1 q^{-1/2} x_i}{1-\epsilon_1 q^{1/2} x_i}) (\epsilon_1 q^{-1/2}-\frac{1-x_i}{1-q x_i}) \nonumber \\
&=\frac{(1-x_i^{-2})(1-qx_i^{-2})}{(1-x_i^2)(1-qx_i^2)}.
\end{align}

Applying the transformation $(\ldots x_{i-2}, x_i, x_{i+2}, \ldots)\mapsto (\ldots x_{i-2} x_i, \frac{1}{x_i}, x_i x_{i+2}, \ldots)$ to \ref{residueodd} or \ref{residueeven} just permutes the factors, except for the two factors $(1-x_i^2)^{-1}(1-qx_i^2)^{-1}$, which are replaced by $(1-x_i^{-2})^{-1}(1-qx_i^{-2})^{-1}$. Therefore the formulas of Proposition \ref{Rformula} satisfy the functional equations \ref{resfe}.

Let $R$ and $R'$ be two power series satisfying the functional equations \ref{resfe}. Then the ratio $R/R'$ is invariant under $(\ldots x_{i-2}, x_i, x_{i+2}, \ldots)\mapsto (\ldots x_{i-2} x_i, \frac{1}{x_i}, x_i x_{i+2}, \ldots)$ for $i$ even and, if $n$ is even, $0<i<n$. In the case of $n$ odd, this immediately implies that $R/R'$ is a diagonal power series, so the formula \ref{residueodd} is correct up to diagonal factors.

If $n$ is even, we require additional functional equations. There are two functional equations corresponding to the transformations $(\sigma_0\sigma_1\cdots\sigma_n)^2$ and $(\sigma_n\sigma_{n-1}\cdots\sigma_0)^2$. We will describe the $(\sigma_0\sigma_1\cdots\sigma_n)^2$ functional equation--the other one is similar. We have
\begin{align}
&R(x_0, x_2, x_4, \ldots x_{n-4}, x_{n-2}, x_n) \nonumber \\ &=(*) R(x_0^3x_2^3x_4^2\cdots x_n^2, x_4, x_6, \ldots x_{n-2}, x_0x_n, x_0^{-3}x_2^{-2}\cdots x_n^{-2}) 
\end{align}
where 
\begin{equation}
(*)=\frac{1-q^{1/2}(x_0x_2\cdots x_n)^{-4} x_0^{-2}}{1-q^{1/2}(x_0x_2\cdots x_n)^4 x_0^2} \prod\limits_{m=0}^1 \prod\limits_{\substack{0\leq i<n \\ \text{even}}} \frac{(1-q^{1/2}(x_0x_2\cdots x_n)^{-2m}(x_0x_2\cdots x_i)^{-2})}{(1-q^{1/2}(x_0x_2\cdots x_n)^{2m}(x_0x_2\cdots x_i)^2)}
\end{equation}
(the transformation is slightly different when $n=2$).

Finally, suppose $n$ is even and $n>2$. Consider the transformation $\sigma_0\sigma_1\sigma_n\sigma_{n-1}\sigma_0\sigma_n\sigma_1\sigma_0$. This leads to a functional equation 
\begin{align}
&R(x_0, x_2, x_4, \ldots x_{n-4}, x_{n-2}, x_n) \nonumber \\ &=(*) R(x_n^{-1}, x_0x_2x_n, x_4, \ldots x_{n-4}, x_0x_{n-2}x_n, x_0^{-1}) 
\end{align}
where 
\begin{equation}
(*)=\frac{(1-q^{1/2}x_0^{-2})(1-q^{1/2}x_n^{-2})(1-x_0^{-2}x_n^{-2})(1-qx_0^{-2}x_n^{-2})}{(1-q^{1/2}x_0^2)(1-q^{1/2}x_n^2)(1-x_0^2x_n^2)(1-qx_0^2x_n^2)}
\end{equation}
(the functional equation is slightly different when $n=4$). 

It is straightforward to show that the formula \ref{residueeven} satisfies these additional functional equations, and that they determine it up to diagonal factors. This completes the computation of $R$ up to diagonal factors. Note that the off-diagonal part $R_0$ satisfies Property \ref{symmetry}. 

With the off-diagonal factors in hand, we can compute all of $R$. In fact, by Property \ref{symmetry} we need only compute $R^{\flat}$ and $R^{\natural}$. For $n$ odd, we must show:
\begin{align}\label{residueoddflat}
&R^{\flat}(x_0 x_2 \cdots x_{n-1})= \nonumber \\ 
&=\prod_{m=0}^{\infty} (1-(x_0 x_2\cdots x_{n-1})^{2m+1})^{-1} \prod_{\substack{i,j \text{ mod } n+1 \\ \text{even}}} (1-(x_0 x_2\cdots x_{n-1})^{2m}(x_i x_{i+2}\cdots x_j)^2)^{-1}
\end{align}
where only the first factor and the factors with $j\equiv i-2$ are not yet determined. For $n$ even, we must show:
\begin{align} \label{residueevenflat}
&R^{\flat}(x_0, x_2, \ldots x_n)= \nonumber \\
&=\prod_{m=0}^{\infty} \prod_{\substack{i,j \text{ mod } n+2 \\ \text{even} \\i\neq 0, \, j\neq n}} (1-(x_0x_2\cdots x_n)^{2m}(x_ix_{i+2}\cdots x_j)^2)^{-1}
\end{align}
where only the factors with $j \equiv i-2$ are not yet determined. Further, $R^{\natural}$ has no diagonal factors.

Before starting the computation, we state a useful lemma on the combinatorics of the partition function. 

\begin{lemma}
The generating function of integer partitions $\delta^{(0)}\geq \delta^{(1)}\geq \delta^{(2)}\geq\ldots$ such that $\sum\limits _{j\equiv k \mod n} \delta^{(j)} =a_k$ is 
\begin{equation}
\prod_{m=0}^{\infty} \prod_{j=1}^{n} (1-(x_1x_2\cdots x_j)(x_1x_2\cdots x_n)^m)^{-1}
\end{equation}
\end{lemma}

As a consequence, the generating function of $n$-tuples of integer partitions $\delta_1^{(j)}, \delta_2^{(j)}, \ldots  \delta_{n}^{(j)}$ such that $\sum\limits _{i+j \equiv k \mod n} \delta_i^{(j)} =a_k$ is
\begin{equation} \label{example}
\prod_{m=0}^{\infty} \prod_{i, j \text{ mod } n} (1-(x_ix_{i+1}\cdots x_j)(x_1x_2\cdots x_n)^m)^{-1}.
\end{equation}

The next proposition establishes, in the case of $n$ odd, that the diagonal part of $R^{\flat}$ matches the diagonal part of Equation \ref{residueoddflat}. Recall that 
\begin{equation}
P(x)Z_{\text{diag}}(q^{-(n+1)/2}x)=R_{\text{diag}}(x)
\end{equation}
and $Z_{\text{diag}}(q^{-(n+1)/2}x)=1+O(q)$. Here $P(x)=\sum\limits_a p_a(q)x^{a}$ and $p_a(q)$ is the value of $q^{-a(n+1)}c_{a, 2a, \ldots a, 2a}(q)$ in a series satisfying the functional equations whose diagonal coefficients are $c_{0, 0, \ldots 0}=1$, $c_{a, a, \ldots a}=0$ for all $a>0$. We explicitly compute the part of $P(x)$ which has degree $0$ in the variable $q$. 

\begin{prop}\label{oddcombinatorics}
Let $n$ be odd. Then
\begin{align}
&P(x_0 x_2\cdots x_{n-1}) \nonumber \\
&=\bigg{(}\prod_{m=0}^{\infty} (1-(x_0 x_2\cdots x_{n-1})^{2m+1})^{-1} \nonumber \\
&\hspace{3em} \prod_{\substack{i,j \text{ mod } n+1 \\ \text{even}}} (1-(x_0 x_2\cdots x_{n-1})^{2m}(x_i x_{i+2}\cdots x_j)^2)^{-1}\bigg{)}_{\text{diag}} + O(q).
\end{align}
\end{prop}

This implies that
\begin{align}
&R^{\flat}_{\text{diag}}(x_0 x_2\cdots x_{n-1}) \nonumber \\
&=\bigg{(}\prod_{m=0}^{\infty} (1-(x_0 x_2\cdots x_{n-1})^{2m+1})^{-1} \nonumber \\
&\hspace{3em} \prod_{\substack{i,j \text{ mod } n+1 \\ \text{even}}} (1-(x_0 x_2\cdots x_{n-1})^{2m}(x_i x_{i+2}\cdots x_j)^2)^{-1}\bigg{)}_{\text{diag}}.
\end{align}

\begin{proof}
The proof requires closely examining the combinatorics of the recurrences on coefficients of $Z$. Recall the statement of the recurrence associated to functional equation $\sigma_i$: for $a_{i-1}+a_{i+1}$ odd,
\begin{equation} 
c_{\ldots a_i, \ldots}=q^{a_i-(a_{i-1}+a_{i+1}-1)/2} c_{\ldots a_{i-1}+a_{i+1}-1-a_i, \ldots}
\end{equation}
and for $a_{i-1}+a_{i+1}$ even, applying equation \ref{evenrecurrence} repeatedly gives
\begin{equation}
c_{\ldots a_i, \ldots}=q^{a_i-(a_{i-1}+a_{i+1})/2}(c_{\ldots (a_{i-1}+a_{i+1})/2, \ldots} + \sum_{a_i'=a_{i-1}+a_{i+1}-a_i}^{(a_{i-1}+a_{i+1})/2-1} (c_{\ldots a_i', \ldots} -q c_{\ldots a_i'-1, \ldots})).
\end{equation}

Starting with $c_{a_0, \ldots a_n}$ we will apply the recurrences in the following order: first, reduce as far as possible with the odd $\sigma_i$, then reduce the result as far as possible with the even $\sigma_i$, then reduce that result as far as possible with the odd $\sigma_i$, and so on. Any coefficient will eventually be reduced to a linear combination of diagonal coefficients in this way. The lowest term in $p_a(q)$ represents the number of paths from $c_{a, 2a, a, 2a, \ldots, a, 2a}$ to $c_{0, 0,\ldots 0}$ via these recurrences, gaining as small a power of $q$ as possible. 

Given any $c_{a_0, \ldots a_n}$, assuming without loss of generality that $\sum\limits_{i \text{ odd}} a_i \geq \sum\limits_{i \text{ even}} a_i$, we apply the recurrences $\sigma_i$ for $i$ odd to reduce as far as possible. Any coefficient $c_{a_0, a_1', \ldots a_{n-1}, a_n'}$ in the resulting expression now has $\sum\limits_{i \text{ odd}} a_i' \leq \sum\limits_{i \text{ even}} a_i$. Furthermore, it is multiplied by a factor of at least $q^{\sum\limits_{i \text{ odd}} a_i - \sum\limits_{i \text{ even}} a_i}$, and more than this if any of the $a_i$ for $i$ odd could not be reduced. If we continue reducing this way until we reach $c_{0, 0, \ldots 0}$, it will be multiplied by a factor of at least $q^{\text{Max}(\sum\limits_{i \text{ odd}} a_i, \sum\limits_{i \text{ even}} a_i)}$. In particular, reducing $c_{a, 2a, a, 2a, \ldots, a, 2a}$ to $c_{0, 0,\ldots 0}$ involves multiplying by at least $q^{a(n+1)}$. This is the correct order since one possible path is 
\begin{equation}
c_{a, 2a, a, 2a, \ldots a, 2a} \to q^{a(n+1)/2} c_{a, 0, a, 0, \ldots a, 0} \to q^{a(n+1)} c_{0, 0, \ldots 0}.
\end{equation}
Therefore $p_a(q)$ is a polynomial in $q$ with nonzero constant term.

Because we are only considering the lowest term in $p_a(q)$, we can discard all terms in the $\sigma_i$ recurrence with a factor greater than $q^{a_i-(a_{i-1}+a_{i+1})/2}$. This leads to greatly simplified recurrences: if $a_{i-1}+a_{i+1}$ is even, then 
\begin{equation} \label{evenrec}
c_{\ldots a_i, \ldots}=q^{a_i-(a_{i-1}+a_{i+1})/2}\sum_{a_i'=a_{i-1}+a_{i+1}-a_i}^{(a_{i-1}+a_{i+1})/2} c_{\ldots a_i', \ldots}
\end{equation}
and if $a_{i-1}+a_{i+1}$ is odd, then
\begin{equation} \label{oddrec}
c_{\ldots a_i, \ldots}=0.
\end{equation}

We have now reduced the problem of computing the constant coefficient in $p_a(q)$ to counting chains of nonnegative integer indices:
\begin{align*}
&a_0, &&a_1, &&a_2, &&a_3, &&\ldots &&a_{n-1}, &&a_n\\
&a_0', &&a_1', &&a_2', &&a_3', &&\ldots &&a_{n-1}', &&a_n'\\
&a_0'', &&a_1'', &&a_2'', &&a_3'', &&\ldots &&a_{n-1}'', &&a_n''\\
&\cdots &&\cdots &&\cdots &&\cdots &&\cdots &&\cdots &&\cdots\\
&a_0^{(\ell)}, &&a_1^{(\ell)}, &&a_2^{(\ell)}, &&a_3^{(\ell)}, &&\ldots &&a_{n-1}^{(\ell)}, &&a_n^{(\ell)}\\
\end{align*}
such that:
\begin{condition} \label{ends} 
we have the boundary conditions $(a_0, a_1, \ldots a_{n-1}, a_n)=(a, 2a, \ldots a, 2a)$ and $(a_0^{(\ell)}, a_1^{(\ell)}, \ldots a_{n-1}^{(\ell)}, a_n^{(\ell)})=(0, 0 \ldots 0, 0)$
\end{condition}
\begin{condition}
$a_i^{(j)}=a_i^{(j+1)}$ if $i$ is even and $j$ is even, or if $i$ is odd and $j$ is odd.
\end{condition}
\begin{condition} \label{same} 
$a_i^{(j)}+a_{i+2}^{(j)}$ is even for all $i, j$. 
\end{condition}
\begin{condition} \label{ineq} 
For $i$ even and $j$ odd, or $i$ odd and $j$ even, 
\begin{equation}
a_i^j \geq \frac{1}{2}(a_{i-1}^{(j)}+a_{i+1}^{(j)}) \geq a_{i}^{(j+1)} \geq a_{i-1}^{(j)}+a_{i+1}^{(j)}-a_i^{(j)}.
\end{equation}
\end{condition}
Note that the indices $i$ are still numbered modulo $n+1$ here.

We will rephrase this counting problem once before solving it. For $i$ modulo $n+1$ even, and $1 \leq j \leq \ell$, let $\delta_i^{(j)}=a_{i+j-1}^{(j-1)}-a_{i+j-2}^{(j)}$. The last inequality of condition (\ref{ineq}) implies that 
\begin{equation}
\delta_i^{(1)}\geq\delta_i^{(2)}\geq\delta_i^{(3)}\geq\cdots\geq \delta_i^{(\ell)}\geq 0
\end{equation}
Thus a chain of indices as above gives rise to an $(n+1)/2$-tuple of integer partitions $\delta_0^{(j)}, \delta_2^{(j)}, \ldots \delta_{n-1}^{(j)}$ satisfying the following two conditions, which correspond to conditions \ref{same} and \ref{ends}.
\begin{condition}
For fixed $j$, the $\delta_i^{(j)}$ are either all even or all odd.
\end{condition}
\begin{condition} \label{tele} 
$\sum\limits_j \delta_{i-2j}^{(j)}=a$ for all $i$.
\end{condition}

We can reconstruct the chain of indices $a_i^{(j)}$ from the partitions $\delta_i^{(j)}$: for $i$, $j$ both odd or both even, \begin{equation}
a_i^{(j)}=\sum\limits_{k=j+1}^\ell \delta_{i+j+2-2k}^{(k)}.
\end{equation}
By comparing such expressions, we find that for any $i$, $j$ both odd or both even, $a_i^{(j)}\geq a_{i-1}^{(j+1)}$ and $a_i^{(j)}\geq a_{i+1}^{(j+1)}$. This appears to be a stronger condition than \ref{ineq}; in fact, it must be equivalent.

To count the $\delta_i^{(j)}$, first note that there exists a unique strictly decreasing partition $\gamma$ such that for all $i$, there exists a partition $\tilde{\delta}_i$ with even entries such that $\delta_i=\tilde{\delta}_i + \gamma^*$. Here $*$ denotes the conjugate partition. We may take $\gamma$ to be the set $\lbrace j : d_i^{(j)} \text{ odd}\rbrace$, in decreasing order. If $\gamma_1$ and $\gamma_2$ have this same property, then $\gamma_1^*+\gamma_2^*$ has all even entries, and, since $\gamma_1$ and $\gamma_2$ are strictly decreasing, this implies that they are equal. 

Since the generating function of strictly decreasing partitions is the same as the generating function of odd partitions, $\prod\limits_{k=0}^{\infty}(1-x^{2k+1})^{-1}$, the first factor of \ref{residueoddflat} will account for the choice of $\gamma$.

Now it suffices to count $(n+1)/2$-tuples of even partitions $\tilde{\delta}_i^{(j)}$ satisfying condition (\ref{tele}). But by the logic of equation \ref{example}, this is precisely the diagonal part of the second factor of \ref{residueoddflat}.
\end{proof}

This completes the computation, and the proof of the main theorem, in the case of $n$ odd. In the case of $n$ even, the proof is similar, supplemented by a lemma which does not hold in the odd case.

\begin{lemma}
Suppose $n$ is even. Then $P(x)$ is an even power series in $x$.
\end{lemma}
\begin{proof}
A list of indices $a_0, a_1, \ldots a_n$ can be broken into blocks of consecutive even or odd indices. The list $a, 2a, a, 2a, \ldots a$ for $a$ odd has a single even-length block of odd indices: the first and last $a$. Since the recurrences can only change the parity of an index if the sum of its neighbors is even, they preserve the property of having an odd number of even-length blocks of odd indices. In particular, there is no path from $a, 2a, a, 2a, \ldots a$ to $0, 0, \ldots 0$ via the recurrences.
\end{proof}

Since the off-diagonal factors of $R$ are all even in $x_0, x_2, \ldots x_n$ and $P(x)$ is even, the diagonal factors of $R^{\flat}R^{\natural}$ must be even as well. Thus $R$ is an even power series, which is not obvious a priori. In particular, $R^{\natural}$ cannot contain diagonal factors, so it suffices to describe the diagonal part of $R^{\flat}$. The following proposition is the analogue of \ref{oddcombinatorics}. The proof is parallel, so many details are omitted. 

\begin{prop}
Let $n$ be even. Then 
\begin{align}
&P(x_0x_2\cdots x_n)= \nonumber \\
&=\left(\prod_{m=0}^{\infty} \prod_{\substack{i,j \text{ mod } n+2 \\ \text{even} \\i\neq 0, \, j\neq n}} (1-(x_0x_2\cdots x_n)^{2m}(x_ix_{i+2}\cdots x_j)^2)^{-1}\right)_{\text{diag}}+O(q^{1/2}).
\end{align}
\end{prop}

This implies that 
\begin{align}
&R^{\flat}_{\text{diag}}(x_0x_2\cdots x_n)= \nonumber \\
&=\left(\prod_{m=0}^{\infty} \prod_{\substack{i,j \text{ mod } n+2 \\ \text{even} \\i\neq 0, \, j\neq n}} (1-(x_0x_2\cdots x_n)^{2m}(x_ix_{i+2}\cdots x_j)^2)^{-1}\right)_{\text{diag}}
\end{align}
which completely determines $R$.

\begin{proof}
We apply the $\sigma_i$ recurrences to $c_{a, 2a, a, 2a,\ldots a}(q)$ in the following order: first, apply $\sigma_1\sigma_3\cdots\sigma_{n-1}$, then $\sigma_2\sigma_4\cdots \sigma_{n}$, then $\sigma_3\sigma_5\cdots\sigma_{n+1}$ (recall that $\sigma_{n+1}=\sigma_0$), etc. Eventually, every index will be reduced to zero by these recurrences. 

If we apply the recurrences $\sigma_1\sigma_3\cdots\sigma_{n-1}$ to an arbitrary $c_{a_0, a_1,\ldots a_n}$, we obtain a linear combination of lower coefficients $c_{a_0, a_1', a_2, a_3', \ldots a_n}$, each of which is multiplied by $q$ to the power of at least $a_1+a_3+\cdots+a_{n-1}-(a_2+a_4+\cdots+a_{n-2})-\frac{1}{2}(a_0+a_n)$, and more than this if any of the $a_1, a_3, \ldots a_{n-1}$ could not be reduced. Repeating this process with $\sigma_2 \sigma_4\cdots\sigma_n$ and so on until we reach $c_{0, 0, \ldots 0}$, we gain a factor of at least $q^{a_1+a_3+\cdots+a_{n-1}+a_n/2}$. This lower bound is actually the correct order for the coefficient $c_{a, 2a, a, 2a,\ldots a}$ with $a$ even, since one possible path is 
\begin{equation}
c_{a, 2a, a, 2a,\ldots a} \to q^{a(n-1)/2}c_{a, 0, a, 0,\ldots a} \to q^{an-3a/2}c_{a, 0, \ldots 0} \to q^{an-a/2}c_{0, \ldots 0}
\end{equation}
Hence $p_a(q)$ is a polynomial in $q^{1/2}$ with nonzero constant coefficient.

Once again, any term in the $\sigma_i$ recurrence which carries a power of $q$ greater than $a_i-\frac{1}{2}(a_{i-1}+a_{i+1})$ can be ignored, and we have the simplified recurrences \ref{evenrec} and \ref{oddrec}. 

We must count chains of nonnegative indices:
\begin{align*}
&a_0, &&a_1, &&a_2, &&a_3, &&\ldots &&a_{n-1}, &&a_n \\
&a_0', &&a_1', &&a_2', &&a_3', &&\ldots &&a_{n-1}', &&a_n' \\
&a_0'', &&a_1'', &&a_2'', &&a_3'', &&\ldots &&a_{n-1}'', &&a_n \\
&\cdots &&\cdots &&\cdots &&\cdots &&\cdots &&\cdots &&\cdots \\
&a_0^{(\ell)}, &&a_1^{(\ell)}, &&a_2^{(\ell)}, &&a_3^{(\ell)}, &&\ldots &&a_{n-1}^{(\ell)}, &&a_n^{(\ell)} \\
\end{align*}
such that:
\begin{condition} We have $(a_0, a_1, a_2, a_3, \ldots a_{n-1}, a_n)=(a, 2a, a, 2a, \ldots, 2a, a)$ and $(a_0^{(\ell)}, a_1^{(\ell)}, a_2^{(\ell)}, a_3^{(\ell)}, \ldots a_{n-1}^{(\ell)}, a_n^{(\ell)})=(0,0,0,0,\ldots 0,0)$.
\end{condition}
\begin{condition} All $a_i^{(j)}$ are even. (This is the analogue of condition \ref{same} when $n$ is even.)
\end{condition}
\begin{condition} If $i\in \lbrace j, j+2, j+4, \ldots j+n \rbrace$, then $a_{i}^{(j)}=a_{i}^{(j+1)}$.
\end{condition}
\begin{condition}\label{evenineq} If $i\in \lbrace j+1, j+3, j+5, \ldots j+n-1 \rbrace$, then 
\begin{equation} 
a_{i}^{(j)}\geq \frac{1}{2}(a_{i-1}^{(j)}+a_{i+1}^{(j)})\geq a_{i}^{(j+1)}\geq a_{i-1}^{(j)}+a_{i+1}^{(j)}-a_{i}^{(j)}
\end{equation}
\end{condition}

To further simplify, for $i \in \lbrace 2, 4, 6, \ldots n \rbrace$ and $1 \leq j \leq \ell$, let $\delta_i^{(j)}=a_{i+j-1}^{(j-1)}-a_{i+j-2}^{(j)}$. Then by condition \ref{evenineq} above, we have $\delta_i^{(1)}\geq \delta_i^{(2)} \geq \delta_i^{(3)} \geq \cdots \geq \delta_i^{(\ell)} \geq 0$. To simplify notation, we also set $\delta_0^{(j)}=0$ for all $j$, and we take the lower index $i$ of $\delta_i$ modulo $n+2$ instead of $n+1$. We are now counting $\frac{n}{2}+1$-tuples of integer partitions $\delta_0, \delta_2, \delta_4, \ldots \delta_n$ such that:
\begin{condition} All $\delta_i^{(j)}$ are even.
\end{condition}
\begin{condition} $\delta_0^{(j)}=0$
\end{condition}
\begin{condition} $\sum_{j=1}^{\ell} \delta_{i-2j}^{(j)}=a$ for all $i$.
\end{condition}
The indices $a_i^{(j)}$ can all be recovered from the partitions $\delta_i$ satisfying these conditions.

By the logic of Equation \ref{example}, the generating function of such sets of partitions is the diagonal part of the series:
\begin{equation}
\prod_{m=0}^{\infty} \prod_{\substack{i,j \text{ mod } n+2 \\ \text{even} \\i\neq 0}} (1-(x_0x_2\cdots x_n)^{2m}(x_ix_{i+2}\cdots x_j)^2)^{-1}
\end{equation}
which is the same as the diagonal part of 
\begin{equation}
\prod_{m=0}^{\infty} \prod_{\substack{i,j \text{ mod } n+2 \\ \text{even} \\i\neq 0, \, j\neq n}} (1-(x_0x_2\cdots x_n)^{2m}(x_ix_{i+2}\cdots x_j)^2)^{-1}.
\end{equation}
\end{proof}

This completes the proof of the main theorem in the case of $n$ even.

\bibliographystyle{amsplain}
\bibliography{MDSbib}

\providecommand{\bysame}{\leavevmode\hbox to3em{\hrulefill}\thinspace}
\providecommand{\MR}{\relax\ifhmode\unskip\space\fi MR }
\providecommand{\MRhref}[2]{%
  \href{http://www.ams.org/mathscinet-getitem?mr=#1}{#2}
}
\providecommand{\href}[2]{#2}
\begin{thebibliography}{10}

\bibitem{BGKP}
A.~Braverman, H.~Garland, D.~Kazhdan, and M.~Patnaik, \emph{An affine
  {G}indikin-{K}arpelevic formula}, preprint.

\bibitem{BK}
A.~Braverman and D.~Kazhdan, \emph{The spherical {H}ecke algebra for affine
  {K}ac-{M}oody groups {I}}, Ann. of Math \textbf{(2) 174} (2011), no.~3,
  1603--1642.

\bibitem{BBF1}
B.~Brubaker, D.~Bump, and S.~Friedberg, \emph{Weyl group multiple {D}irichlet
  series, {E}isenstein series and crystal bases}, Ann. of Math. \textbf{(2)
  173} (2011), no.~2, 1081--1120.

\bibitem{BBF2}
\bysame, \emph{Weyl group multiple {D}irichlet series: Type {A} combinatorial
  theory}, Ann. of Math. Studies, vol. 175, Princeton University Press,
  Princeton, NJ, 2011.

\bibitem{BD}
A.~Bucur and A.~Diaconu, \emph{Moments of quadratic {D}irichlet {L}-functions
  over rational function fields}, Moscow Math. J. \textbf{10} (2010), no.~3,
  485--517.

\bibitem{BFH3}
D.~Bump, S.~Friedberg, and J.~Hoffstein, \emph{On some applications of
  automorphic forms to number theory}, Bulletin of the Amer. Math. Soc.
  \textbf{33} (1996), no.~2, 157--175.

\bibitem{CG1}
G.~Chinta and P.E. Gunnells, \emph{Weyl group multiple {D}irichlet series
  constructed from quadratic characters}, Invent. Math. \textbf{167} (2007),
  no.~2, 327--353.

\bibitem{CG2}
\bysame, \emph{Constructing {W}eyl group multiple {D}irichlet series}, J. Amer.
  Math. Soc. \textbf{23} (2010), no.~1, 189--215.

\bibitem{DGH}
A.~Diaconu, D.~Goldfeld, and J.~Hoffstein, \emph{Multiple {D}irichlet series
  and moments of zeta and {L}-functions}, Compositio Math. \textbf{139} (2003),
  no.~3, 297--360.

\bibitem{DP}
A.~Diaconu and V.~Pasol, \emph{Trace formulas, character sums, and multiple
  {D}irichlet series}, preprint.

\bibitem{FZ}
S.~Friedberg and L.~Zhang, \emph{Eisenstein series on covers of odd orthogonal
  groups}, preprint.

\bibitem{G}
H.~Garland, \emph{Certain {E}isenstein series on loop groups: convergence and
  the constant term}, Algebraic Groups and Arithmetic, Tata Inst. Fund. Res.,
  Mumbai, 2004, pp.~275--319.

\bibitem{GMP}
H.~Garland, M.~Patnaik, and S.~Miller, \emph{Entirety of cuspidal {E}isenstein
  series on loop groups}, preprint.

\bibitem{GH}
D.~Goldfeld and J.~Hoffstein, \emph{Eisenstein series of 1/2 integral weight
  and the mean value of real {D}irichlet {L}-series}, Invent. Math. \textbf{80}
  (1985), no.~2, 185--208.

\bibitem{LZ}
K.H. Lee and Y.~Zhang, \emph{Weyl group multiple {D}irichlet series for
  symmetrizable {K}ac-{M}oody root systems}, preprint.

\bibitem{P}
M.~Patnaik, \emph{Unramified {W}hittaker functions on p-adic loop groups},
  preprint.

\bibitem{W}
I.~Whitehead, \emph{Multiple {D}irichlet series for affine {W}eyl groups},
  Ph.D. thesis, Columbia University, 2014.

\end{thebibliography}

\end{document}